\documentclass[12pt]{article}
\usepackage{url}
\usepackage{amsfonts,amsmath,amsthm,amssymb}
\usepackage{comment}
\newtheorem{lem}{Lemma}[section]
\newtheorem{prop}[lem]{Proposition}
\newtheorem{thm}[lem]{Theorem}

\newtheorem{cor}[lem]{Corollary}
\newtheorem{main}{Theorem}
\theoremstyle{remark}
\newtheorem{rem}[lem]{Remark}
\theoremstyle{definition}

\newcommand{\FF}{\mathbb{F}}

\newcommand{\cB}{\mathcal{B}}
\newcommand{\cD}{\mathcal{D}}
\newcommand{\cP}{\mathcal{P}}
\newcommand{\allone}{\mathbf{1}}

\DeclareMathOperator{\wt}{wt}
\DeclareMathOperator{\supp}{supp}

\allowdisplaybreaks

\usepackage{lscape}

\begin{document}
\title{On $t$-designs supported by self-orthogonal codes}

\author{
Masaaki Harada\thanks{
Research Center for Pure and Applied Mathematics,
Graduate School of Information Sciences,
Tohoku University, Sendai 980--8579, Japan.
email: \protect\url{mharada@tohoku.ac.jp}
}
}

\maketitle

\begin{abstract}
We study self-orthogonal $t$-$(v,k,\lambda)$ designs
satisfying the condition that the codes generated by the designs
are self-orthogonal codes whose
codewords of minimum weight $k$ correspond to the blocks.
A classification of such designs is given
under the conditions that $t\geq \lfloor k/4\rfloor+1\geq2$ and $t \le 11$.
This yields a classification of 
self-orthogonal codes of minimum weight $k$
whose codewords of minimum weight form a $t$-design under the
same conditions.
\end{abstract}

\section{Introduction}\label{sec:1}

A $t$-$(v,k,\lambda)$ design $\cD$ is a pair of a set $\cP$ of $v$ points
and a collection $\cB$ of $k$-element subsets of 
$\cP$ (called blocks)
such that every $t$-element subset of $\cP$ is contained in exactly
$\lambda$ blocks, where $\lambda$ is a positive integer.
As usual, we assume that $t < k < v$, throughout this paper.
Hereafter, all designs are assumed to have no repeated blocks.
Often a $t$-$(v,k,\lambda)$ design is simply called a $t$-design.
A $t$-design can be represented by its (block-point)
incidence matrix $A=(a_{ij})$, where $a_{ij}=1$ if the $j$th
point is contained in the $i$th block
and $a_{ij}=0$ otherwise.
Two designs with the same parameters are \emph{isomorphic} if
there is a bijection between their point sets that maps the blocks
of the first design into blocks of the second design.

A $t$-$(v,k,\lambda)$ design $\cD$ is called \emph{self-orthogonal} if
the block intersection numbers have the same parity
as the block size $k$ (see~\cite{T86}).
The term self-orthogonal is due to a natural connection between 
such designs and self-orthogonal codes.
If $\cD$ is a self-orthogonal $t$-$(v,k,\lambda)$ design
with $k$ even, then the code $C(\cD)$
generated by the rows of an incidence matrix of $\cD$
is self-orthogonal.
Throughout this paper, we assume that
$k$ is even for a self-orthogonal $t$-$(v,k,\lambda)$ design, 
unless otherwise specified.

We define 
a \emph{saturated} design to be a self-orthogonal $t$-$(v,k,\lambda)$
design $\cD=(\cP,\cB)$ with $k$ even, satisfying the
following conditions:
\begin{align}
&\text{$C(\cD)$ is self-dual, that is, }
C(\cD)=C(\cD)^\perp, \label{C1} \\
&\text{$\cB$ is the set of supports of the codewords of 
weight $k$ in $C(\cD)$,} \label{C2} \\
&\text{the minimum weight of $C(\cD)$ is $k$,} \label{C3} 
\end{align}
where $C(\cD)^\perp$ denotes the dual code of $C(\cD)$.
A self-orthogonal
$t$-$(v,k,\lambda)$ design with $k$ even
failing at least one of~\eqref{C1}, \eqref{C2} and \eqref{C3}
is called \emph{unsaturated}.
It is well known that the Witt system $\cD$ on $24$ points, that
is, the $5$-$(24,8,1)$ design is saturated, where
$C(\cD)$ is the extended Golay code $G_{24}$.

The only known extremal doubly even
self-dual codes of length $24\ell$
are the extended Golay code $G_{24}$ and the
extended quadratic residue code $QR_{48}$ of length $48$.
In addition, it is known that $G_{24}$ (resp.\ $QR_{48}$)
is the unique extremal doubly even self-dual code of length
$24$ (resp.\ $48$).
The existence of an extremal doubly even self-dual
code of length $72$ is a long-standing open question.
By the Assmus--Mattson theorem~\cite{AM},
the supports of the codewords of minimum weight in
an extremal doubly even self-dual code of length $24\ell$
form a self-orthogonal $5$-$(24\ell,4\ell+4,\binom{5\ell-2}{\ell-1})$ design~\cite{MS73}.
As described above,
it is well known that the Witt system on $24$ points is saturated.
For $\ell=2,3$ and $4$, every self-orthogonal
$5$-$(24\ell,4\ell+4,\binom{5\ell-2}{\ell-1})$ design is
saturated~\cite{HMT05}, \cite{HKM04} and \cite{H05}, respectively.
A powerful tool used in~\cite{HKM04} and \cite{HMT05}
is the system of fundamental equations,
sometimes called the Mendelsohn equations~\cite{CW12}
(see also~\cite{M71} and \cite{T86}).

The purpose of this paper is to begin a systematic
study of the family of saturated self-orthogonal designs.
We show that self-orthogonal designs with certain parameters
are saturated,
by proving the nonexistence of unsaturated designs with the same parameters.
As a typical example, we have the following:

\begin{main}\label{thm:24-96}
 For $\ell=1,2$ and $3$, every self-orthogonal
$(\ell+2)$-$(24\ell,4\ell+4,\lambda)$ design is 
a saturated $5$-$(24\ell,4\ell+4,\binom{5\ell-2}{\ell-1})$ design.
\end{main}

\begin{rem}\label{rem:24-48}
\begin{enumerate}
\item 
Every self-orthogonal
$3$-$(24,8,\lambda)$ design is isomorphic to the Witt system
on $24$ points.
Every self-orthogonal
$4$-$(48,12,\lambda)$ design is isomorphic to the unique self-orthogonal
$5$-$(48,12,8)$ design.
\item 
The divisibility condition shows that 
none of the $5$-$(24\ell,4\ell+4,\binom{5\ell-2}{\ell-1})$ designs
for $\ell=1,2$ and $3$ is a $6$-design. We refer to the proof of
Theorem~5 in~\cite{BKST} for general $\ell$.
\end{enumerate}
\end{rem}

In Table~\ref{Tab:code}, we give some necessary notation for codes,
which is used throughout this paper.

%%%%%%%%%%%%%%%%%%%%%%%%%%%
\begin{table}[thb]
\caption{Notation for some codes}
\label{Tab:code}
\begin{center}
{\small
%{\footnotesize
%{\scriptsize
\begin{tabular}{c|l}
\noalign{\hrule height0.8pt}
Notation & \multicolumn{1}{c}{Codes} \\
\hline
$e_7$ & the Hamming $[7,4,3]$ code \\
$e_8$ & the extended Hamming $[8,4,4]$ code \\
$R(1,4)$ & the first order Reed--Muller $[16,5,8]$ code \\
$C_{22}$ & the unique (singly even) self-dual $[22,11,6]$ code (\cite{PS75}) \\
$G_{23}$ & the Golay $[23,12,7]$ code \\
$G_{24}$ & the extended Golay $[24,12,8]$ code \\
     %which is obtained by subtraction from the extended Golay
$C_{24}$ & the unique (singly even) self-dual $[24,12,6]$ code (\cite{PS75})\\ 
$C_{32,1}, C_{32,2},\ldots, C_{32,5}$ & the five doubly even 
self-dual $[32,16,8]$ codes (\cite{CPS})\\
$C_{46}$ & the unique (singly even) self-dual $[46,23,10]$ code (\cite{HMT05})\\
$QR_{47}$ & the quadratic residue $[47,24,11]$ code \\
$QR_{48}$ & the extended quadratic residue $[48,24,12]$ code \\
$C_{72}$ &a putative doubly even self-dual $[72,36,16]$ code \\
\noalign{\hrule height0.8pt}
\end{tabular}
}
\end{center}
\end{table}
%%%%%%%%%%%%%%%%%%%%%%%%%%%%%%%%%%%%%%%%%%%%%%%%

Lalaude-Labayle~\cite{Labayle} gave a classification of
self-orthogonal codes of minimum weight $k$
whose codewords of minimum weight form a $t$-design
for the case $t=3$ with $k \le 10$ and the case $t=5$  with $k \le 18$.

\begin{thm}[Lalaude-Labayle~\cite{Labayle}]
\label{thm:Labayle}
Let $C$ be a self-orthogonal code of minimum weight $k$
whose codewords of minimum weight form a $t$-design.
\begin{enumerate}
\item
If $t=3$ and $k \le 10$, then
$C$ is equivalent to one of
$e_8$,
$R(1,4)$,
$(C_{22})_0$,
$C_{22}$, $G_{23}^\perp$, $G_{24}$,
$C_{32,1},C_{32,2},\ldots,C_{32,5}$ and $C_{46}$,
where $(C_{22})_0$ is the doubly even $[22,10,8]$  subcode of $C_{22}$.

\item
If $t=5$ and $k \le 18$, then
$C$ is equivalent to one of
$G_{24}$, $QR_{48}$ and $C_{72}$.
\end{enumerate}
\end{thm}

The following theorem extends the above theorem,
which is one of the main results of this paper.

\begin{main}\label{thm:code}
Suppose that $t \ge \lfloor k/4\rfloor+1$.
Let $C$ be a self-orthogonal code of minimum weight $k$
whose codewords of minimum weight form a $t$-design.
\begin{enumerate}
\item
If $t=2$, then
$C$ is equivalent to one of $e_7^\perp$, $e_8$,
the unique self-orthogonal $[16, 6, 6]$ code,
the unique shortened $[ 21, 10, 6 ]$ code of $C_{22}$ and
$C_{22}$.
\item
If $t=3$, then $C$ is equivalent to one of 
$e_8$,
$R(1,4)$,
$(C_{22})_0$,
$C_{22}$, $G_{23}^\perp$, $G_{24}$,
$C_{32,1},C_{32,2},\ldots,C_{32,5}$ and $C_{46}$,
where $(C_{22})_0$ is the doubly even $[22,10,8]$  subcode of $C_{22}$.
\item
If $t=4$, then
     $C$ is equivalent to one of $G_{23}^\perp$,
     $G_{24}$, $QR_{47}^\perp$ and $QR_{48}$.
\item
      If $t=5$, then $C$ is equivalent to one of $G_{24}$, $QR_{48}$ and $C_{72}$.
\item
If $6 \le t \le 11$, then there is no such code $C$.
\end{enumerate}
 \end{main}

This paper is organized as follows.
%%%
In Section~\ref{sec:pre}, we give definitions, notation
and basic results, and derive lower bounds on the minimum weight
of $C(\cD)^\perp$ using the Mendelsohn equations
for a $t$-$(v,k,\lambda)$ design $\cD$.
%%%
In Section~\ref{sec:4},
we examine the existence of 
self-orthogonal $(\lfloor k/4\rfloor+1)$-$(v,k,\lambda)$ designs
for even $v \le 30$, $k=8,10,12,14$ and $k \le v/2$.
We give strong restrictions on the parameters of 
self-orthogonal designs and
we establish the nonexistence of self-orthogonal designs
for many parameters (Tables~\ref{Tab:Non0} and \ref{Tab:Sum}).
%%%
In Section~\ref{sec:unsaturated},
we give possible parameters of
unsaturated self-orthogonal
$t$-$(v,k,\lambda)$ designs satisfying
$t\geq \lfloor k/4\rfloor+1$ for $v\leq 10000$ and $k=8,10,\ldots,30$
(Table~\ref{Tab:un}).
This is used to show that self-orthogonal designs with certain parameters
are saturated.
As a typical example, we have a proof of
Theorem~\ref{thm:24-96}.
In addition,
the nonexistence of self-orthogonal designs is also established
for many parameters.
In Section~\ref{sec:minwtk}, we study self-orthogonal
designs satisfying~\eqref{C3}.
We give strong restrictions on the parameters of 
self-orthogonal designs satisfying~\eqref{C3} and
$t\geq \lfloor k/4\rfloor+1\geq2$.
This determines possible parameters of
self-orthogonal $t$-$(v,k,\lambda)$ designs with $k$ even
satisfying~\eqref{C3} for $k \le 40$
(Proposition~\ref{prop:2}).
%%%
By investigating self-orthogonal $t$-$(v,k,\lambda)$ designs
with parameters listed in Proposition~\ref{prop:2},
we give a classification of
self-orthogonal $t$-$(v,k,\lambda)$ designs with $k$ even
satisfying~\eqref{C2}, \eqref{C3},
$t\geq \lfloor k/4\rfloor+1\geq2$ and $t \le 11$
in Section~\ref{sec:code}.
Theorem~\ref{thm:code} is obtained as a consequence of
this classification.
As the other consequence, a classification of 
saturated self-orthogonal designs under the
same condition is also obtained (Corollary~\ref{cor:saturated}).

%%%%%%%%%%%%%%%%%%
\section{Preliminaries}\label{sec:pre}

In this section, we give definitions, notation and basic results
used throughout this paper.

\subsection{Self-orthogonal codes}
A (binary) $[n,k]$ \emph{code} $C$ is a $k$-dimensional vector subspace
of $\FF_2^n$,
where $\FF_2$ denotes the finite field of order $2$.
Throughout this paper, all codes are assumed to be binary.
The parameters $n$ and $k$ are called the \emph{length} 
and \emph{dimension} of $C$, respectively.
The \emph{support} $\supp(x)$ of  a vector $x=(x_1,x_2,\ldots,x_n)\in \FF_2^n$ is
$\{i \mid x_i \ne 0\}$ and
the \emph{weight} $\wt(x)$ of $x$ is $|\supp(x)|$.
A vector of $C$ is called a \emph{codeword}.
The minimum nonzero weight of all codewords in $C$ is called
the \emph{minimum weight} of $C$.
An $[n,k,d]$ code is an $[n,k]$ code with minimum weight $d$.
The \emph{weight enumerator} of $C$
is given by the polynomial $\sum_{c \in C}y^{\wt(c)}$.
Two codes are \emph{equivalent} if one can be
obtained from the other by permuting the coordinates.

The \emph{dual} code $C^{\perp}$ of a code
$C$ of length $n$ is defined as
\[
C^{\perp}=
\{x \in \FF_2^n \mid x \cdot y = 0 \text{ for all } y \in C\},
\]
where $x \cdot y$ is the standard inner product.
A code $C$ is called \emph{self-orthogonal}
if $C \subseteq C^\perp$, and $C$ is called \emph{self-dual}
if $C=C^\perp$.
A code is called \emph{doubly even} if all codewords have weight
divisible by $4$.
A doubly even code is self-orthogonal.
It is known that there is a doubly even self-dual code of length $n$
if and only if $n \equiv 0 \pmod 8$.
A self-dual code is called \emph{singly even} if there is
at least one codeword of weight $\equiv 2 \pmod 4$.

Let $\cD$ be a self-orthogonal $t$-$(v,k,\lambda)$ design
with $k$ even.
Let $C(\cD)$ denote the code generated by the rows of an incidence
matrix of $\cD$.
Then $C(\cD)$ is self-orthogonal.
In addition, if $k \equiv 0 \pmod 4$, then
$C(\cD)$ is doubly even.

\subsection{A binomial identity} 

Although the following lemma may be known, we give a proof for completeness.

\begin{lem}\label{lem:kk}
Let $s$ and $j$ be nonnegative integers.
Then
\[
\sum_{i=0}^s (-2)^i\binom{2s-i}{s}\binom{2j+1}{i}
=(-4)^s\binom{j}{s}.
\]
\end{lem}
\begin{proof}
The assertion is immediate for $s=0$.
We proceed by induction on $j$.
If $j=0$, both sides are equal to $\delta_{s,0}$.
Suppose that $j\geq1$ and that the assertion holds for $j-1$.
For $s\geq1$, Pascal's identity gives
\[
\binom{2s-i-1}{s}-\binom{2s-i-2}{s}
=\binom{2s-i-2}{s-1}
\quad (0\leq i\leq s-1).
\]
Using this identity and the induction hypothesis, we obtain
\begin{align*}
&\sum_{i=0}^s (-2)^i\binom{2s-i}{s}\binom{2j+1}{i}\\
&=\sum_{i=0}^s (-2)^i\binom{2s-i}{s}
\left(
\binom{2j-1}{i}
+2\binom{2j-1}{i-1}
+\binom{2j-1}{i-2}
\right)\\
&=(-4)^s\binom{j-1}{s}
-4\sum_{i=0}^{s-1}(-2)^i
\binom{2s-i-1}{s}\binom{2j-1}{i}\\
&\quad
+4\sum_{i=0}^{s-2}(-2)^i
\binom{2s-i-2}{s}\binom{2j-1}{i}\\
&=(-4)^s\binom{j-1}{s}
-4\sum_{i=0}^{s-1}(-2)^i
\left(
\binom{2s-i-1}{s}-\binom{2s-i-2}{s}
\right)\binom{2j-1}{i}\\
&=(-4)^s\binom{j-1}{s}
-4\sum_{i=0}^{s-1}(-2)^i
\binom{2(s-1)-i}{s-1}\binom{2j-1}{i}\\
&=(-4)^s\left(\binom{j-1}{s}+\binom{j-1}{s-1}\right)\\
&=(-4)^s\binom{j}{s}.
\end{align*}
This completes the proof.
\end{proof}

By putting $t=s+1$, the above lemma gives the following
identity:
\begin{equation}\label{eq:kk}
\sum_{i=1}^{t} i (-2)^{i-1}
\binom{2t-i-1}{t-1}\binom{2j}{i}
=0\ (j=1,2,\ldots,t-1),
\end{equation}
which is used in the proof of Lemma~\ref{lem:Mendel}.

\subsection{Mendelsohn equations} 
\label{sec:Mendelsohn}

Let $\cD=(\cP,\cB)$ be a $t$-$(v,k,\lambda)$ design.
It follows that
every $i$-element subset of points ($i\le t$) is contained in exactly
$\lambda_i$ blocks, where
\begin{equation}\label{eq:lambdass}
\lambda_i=\lambda \frac{\binom{v-i}{t-i}}{\binom{k-i}{t-i}}
=\lambda\frac{(v-i)\cdots(v-t+1)}{(k-i)\cdots(k-t+1)}.
\end{equation}
The number $\lambda_1$ is traditionally denoted by $r$, and the total number of
blocks is $b=\lambda_0$.

Assume that $k$ is even.
We denote the code generated by the rows of an 
incidence matrix of $\cD$ by $C(\cD)$.
Let $x$ be a vector of $C(\cD)^{\perp}$ and
let $S$ be the support.
Suppose that $\wt(x)=m$.
Counting the number of pairs $(I,B)$ with $I\subseteq
S\cap B$, $|I|=i$, $B\in\cB$, we have
\begin{equation}
\label{eq:M0}
\sum_{B\in\cB}\binom{|B\cap S|}{i}=\lambda_i\binom{m}{i}
\ (i=0,1,\ldots,t).
\end{equation}
Let
\[
n_j(x)=|\{B\in \cB\mid 2j=|B\cap S|\}|
\ (j=0,1,\ldots).
\]
Note that we often simply denote it by $n_j$.
Then 
we have the following system of equations:
\begin{equation}
\label{eq:M}
\sum_{j=0}^{\lfloor m/2\rfloor}\binom{2j}{i}n_{j}
=\lambda_{i}\binom{m}{i}
\ (i=0,1,\ldots,t).
\end{equation}
These fundamental equations, 
which are sometimes called Mendelsohn equations~\cite{CW12}
(see also~\cite{M71} and \cite{T86}),
are a powerful tool used throughout this paper.

The following lemma gives a necessary condition for the
solvability of systems of the same form as the Mendelsohn
equations~\eqref{eq:M}.

\begin{lem}\label{lem:Mendel}
Let $b_1,b_2,\ldots,b_t$ be nonnegative integers.
If the following system of equations
\[
\sum_{j=1}^{t-1}
\binom{2j}{i}x_{j}
=b_i
\ (i=1,2,\ldots,t)
\]
has a solution $(x_1,x_2,\ldots,x_{t-1})$,
then 
\[
\sum_{i=1}^t k_i b_i =0, 
\]
where
\begin{equation}\label{eq:kidef}
k_i=i (-2)^{i-1}\binom{2t-i-1}{t-1}.
\end{equation}
\end{lem}
\begin{proof}
Using the given equations and~\eqref{eq:kk}, we have
\[
\sum_{i=1}^t k_i b_i
=\sum_{i=1}^t k_i
  \sum_{j=1}^{t-1}\binom{2j}{i}x_j
=\sum_{j=1}^{t-1}x_j
  \sum_{i=1}^t k_i\binom{2j}{i}
=0.
\]
The result follows.
\end{proof}

\subsection{Bounds on minimum weights of $C(\cD)^\perp$}

Throughout this subsection,
let $\cD$ be a $t$-$(v,k,\lambda)$ design with $k$ even.
We give bounds on the minimum weight $d(C(\cD)^\perp)$
of the dual code of $C(\cD)$.
The assertions (i) and  (ii) in the
following lemma were first proved in~\cite{T86} 
for $t=2$ and in~\cite{HMT05} for $t=3$,
respectively.
Our proof is similar to those in~\cite{HMT05} and \cite{T86}.

\begin{lem}\label{lem:dual}
\begin{enumerate}
\item
If $t$ is even with $t \ge 2$, then
$d(C(\cD)^\perp) \ge (t-1) +  \lambda_{t-1}/\lambda$.
\item
If $t$ is odd with $t \ge 3$, then
$d(C(\cD)^\perp) > (t-2) +  \lambda_{t-2}/\lambda_{t-1}$.
\end{enumerate}
\end{lem}

\begin{proof}
The system of equations~\eqref{eq:M} has a solution if we take
$m=d(C(\cD)^\perp)$.
It follows from~\cite[(3.2)]{HMT05} that
\begin{equation}
\label{eq:sum-positive}
\sum_{j=\lfloor \frac{t+1}{2}\rfloor}^{\lfloor \frac{m}{2}\rfloor} n_{j}>0.
\end{equation}

%%%%%%%%%%%%%%%%%
Suppose that $t=2e$, where $e$ is a positive integer.
From \eqref{eq:sum-positive}, we have $m\geq2e$.
Hence, the Mendelsohn equations~\eqref{eq:M} yield the following:
\begin{align*}
\sum_{j=e}^{\lfloor \frac{m}{2}\rfloor} 
2j(2j-1)\cdots(2j-(2e-2))
n_{j} &= m(m-1)\cdots(m-(2e-2)) \lambda_{2e-1}, \\
\sum_{j=e}^{\lfloor \frac{m}{2}\rfloor} 
2j(2j-1)\cdots(2j-(2e-1)) n_{j} &= m(m-1)\cdots(m-(2e-1)) \lambda.
\end{align*}
Hence, we have
\begin{multline*}
m(m-1)\cdots (m-(2e-2))((m-(2e-1))\lambda -  \lambda_{2e-1})
\\
= \sum_{j=e+1}^{\lfloor \frac{m}{2}\rfloor} 
2j(2j-1)\cdots(2j-(2e-2))(2j-2e) n_{j}  \ge 0,
\end{multline*}
where equality holds if and only
if $\sum_{j=e+1}^{\lfloor \frac{m}{2}\rfloor} n_{j} = 0$.
Thus, 
$
m \ge (2e-1) +  \lambda_{2e-1}/\lambda.
$ 
This completes the proof of the assertion (i).

Now suppose that $t=2e+1$, where $e$ is a positive integer.
Since a $t$-$(v,k,\lambda)$ design is a 
$(t-1)$-$(v,k,\lambda_{t-1})$ design,
the above discussion given in the proof of the assertion (i) can be applied.
Thus, it follows that
$d(C(\cD)^\perp) \ge (t-2) +  \lambda_{t-2}/\lambda_{t-1}$, where
equality holds if and only if 
$\sum_{j=e+1}^{\lfloor m/2\rfloor} n_{j} = 0$.
It follows from~\eqref{eq:sum-positive} that $\sum_{j=e+1}^{\lfloor m/2\rfloor} n_{j} > 0$.
The assertion (ii) follows.
\end{proof}

%%%%%%%%%%%%%%%%%%%%%%%%%%%%%%%%%%%%%%%%%%%
\section{Self-orthogonal $t$-$(v,k,\lambda)$ designs with $v$ even}
\label{sec:4}

Throughout this section, we assume that $v$ is even with
$v \le 30$.
We examine the existence of 
self-orthogonal $(\lfloor k/4 \rfloor+1)$-$(v,k,\lambda)$ designs
for $k=8,10,12,14$ and $k \le v/2$.

\begin{lem}\label{3-des-v-2k}
Let $\cD$ be a self-orthogonal $3$-$(v,k,\lambda)$ design such that
$k$ is even and $k \ge 4$.
If there is a vector of weight $4$ in $C(\cD)^{\perp}$,
then $v=2k$.
\end{lem}
\begin{proof}
Let $x$ be a vector of weight $4$ in $C(\cD)^{\perp}$.
From the Mendelsohn equations~\eqref{eq:M}, we have
\begin{align*}
2n_1 +4n_2 &= 4 \lambda\frac{(v-1)(v-2)}{(k-1)(k-2)}, \\
n_1 +6n_2 &= 6 \lambda\frac{v-2}{k-2}\text{ and} \\
4n_2 &= 4 \lambda.
\end{align*}
This implies that $(v-2k)(v-k)=0$.  The result follows.
\end{proof}

Let $d_{\max}(v)$ denote the largest minimum weight
among all self-dual codes of length $v$ (see~\cite[Table~I]{C-S}
for $d_{\max}(v)$ $(v \le 30)$).

\begin{prop}\label{prop:nonex1}
There is no self-orthogonal $3$-$(v,k,\lambda)$ design
for 
\[
(v,k)=(18,  8)\text{ and }(20,  8).
\]
\end{prop}
\begin{proof}
Suppose that $\cD$ is a self-orthogonal
$3$-$(v,k,\lambda)$ design with
$(v,k)=(18,8)$ or $(20,8)$.
By Lemma~\ref{lem:dual} (ii) and $d_{\max}(18)=d_{\max}(20)=4$, 
$C(\cD)$ is contained in some self-dual code with
parameters $[18,9,4]$ and $[20,10,4]$, respectively.
Thus, there is a vector of weight $4$ in $C(\cD)^{\perp}$.
The result follows from Lemma \ref{3-des-v-2k}.
\end{proof}

%%%%%%%%%%%%%%%%%%%%%%%%%%%
\begin{table}[thb]
\caption{Self-orthogonal
$(\lfloor\frac{k}{4}\rfloor+1)$-$(v,k,\lambda)$ designs ($v \le 30, k=8,10,12,14$)}
\label{Tab:Non0}
\begin{center}
{\small
%{\footnotesize
%{\scriptsize
\begin{tabular}{c|c|c|c|rl|c}
\noalign{\hrule height0.8pt}
$(t,v,k)$ & $\tau$ & $\mu_{\max}$& $d(\cD)$ 
& \multicolumn{2}{c|}{$A_{(v,k)}$}
& \multicolumn{1}{c}{$\mu_{BI}$}\\
\hline
$(3, 16,  8)$&   3 & 6& 4      &  198 & (\cite{Pless72})& \\
$(3, 22,  8)$&   6 & 2& 6   &  330 & (\cite{PS75})&  $1$\\
$(3, 24,  8)$&  21 & 1& 6, 8   &  759 &(\cite{PS75})&  \\
$(3, 26,  8)$&   7 & 1& 6      &  390 &(\cite{CPS})& $1$\\
$(3, 28,  8)$&  42 & 0& 6      &  442 &(\cite{CPS})&  \\
$(3, 30,  8)$&   6 & 0& 6      &  393 &(\cite{CPS})& \\
$(3, 20, 10)$&   4 &13& 4      &  512 &(\cite{Pless72})&  $1,2,3,5,6,7,9,10,11,13$\\
$(3, 22, 10)$&   6 & 9& 4, 6   &  700 &(\cite{PS75})&  $1,3,5,7,9$\\
$(3, 24, 10)$& 180 & 0& 4, 6, 8&  960  &(\cite{PS75})& \\
$(3, 26, 10)$&   3 &20& 4, 6   & 1313 &(\cite{CPS})& \\
$(3, 28, 10)$&  20 & 2& 6   & 1624 &(\cite{CPS})& $1,2$\\
$(3, 30, 10)$&  18 & 3& 6   & 1848 &(\cite{CPS})& $1,3$\\
$(4, 24, 12)$&  15 & 8& 6, 8   & 2576 &(\cite{PS75})& $1,2,\ldots,7$\\
$(4, 26, 12)$& 495 & 0& 6      & 2340 &(\cite{CPS})& \\
$(4, 28, 12)$&  99 & 0& 6      & 3717 &(\cite{CPS})& \\
$(4, 30, 12)$&  55 & 1& 6      & 5320 &(\cite{CPS})& $1$\\
$(4, 28, 14)$&  66 & 3& 6      & 5020 &(\cite{CPS})& $1,2,3$\\
$(4, 30, 14)$& 143 & 2& 6      & 8931 &(\cite{CPS})& \\
\noalign{\hrule height0.8pt}
\end{tabular}
}
\end{center}
\end{table}
%%%%%%%%%%%%%%%%%

Let $\cD=(X,\cB)$ be a self-orthogonal $t$-$(v,k,\lambda)$ design with $k$ even.
Then $C(\cD)$ is a self-orthogonal code containing
at least $|\cB|$ codewords of weight $k$.
If $v$ is even, then
$C(\cD)$ is contained in some self-dual $[v,v/2,d]$ code
containing at least $|\cB|$ codewords of weight $k$.
The possible minimum weights $d$ are determined
by Lemma~\ref{lem:dual} and $d_{\max}(v)$.
We denote by $d(\cD)$ the set of the possible minimum weights $d$.
Let $A_{(v,k,d)}$ denote the largest number of codewords of weight $k$ among
self-dual $[v,v/2,d]$ codes.
Let
\[
A_{(v,k)} = \max\{A_{(v,k,d)} \mid d \in d(\cD)\}.
\]
Since $|\cB| \le A_{(v,k)}$,
we have
$\lambda \le 
\lfloor A_{(v,k)}\binom{k}{t}/\binom{v}{t}\rfloor$.
Let $\tau$ be the least positive integer such that
\[
\tau\frac{\binom{v-i}{t-i}}{\binom{k-i}{t-i}}
\]
is an integer for every $i=0,1,\ldots,t$.
Then \eqref{eq:lambdass} implies that $\tau$ divides $\lambda$.
Let
\[
\mu_{\max}=
\left\lfloor
\frac{A_{(v,k)}\binom{k}{t}}{\binom{v}{t} \tau}
\right\rfloor.
\]
Of course, 
if there is a self-orthogonal $t$-$(v,k,\mu \tau)$ design,
then $\mu \le \mu_{\max}$.
In Table~\ref{Tab:Non0}, we list 
$\tau$, $\mu_{\max}$ and $d(\cD)$
for $v \le 30$, $k=8,10,12,14$ and $k \le v/2$, except
$(v,k)=(18,  8)$ and $(20,  8)$.
We also list $A_{(v,k)}$ along with references.
The column  ``$\mu_{BI}$''  is explained below.

Let $\cD=(X,\cB)$ be a self-orthogonal $t$-$(v,k,\lambda)$ design with $k$ even.
Let $B$ be a block of $\cD$.
Let
\[
m_{2j}=|\{B'\in \cB\setminus\{B\} \mid 2j=|B\cap B'|\}|
\ \left(j=0,1,\ldots,\left\lfloor\frac{k}{2}\right\rfloor-1 \right).
\]
From the Mendelsohn equations~\eqref{eq:M},
we have the following system of equations:
\begin{equation}\label{eq:BI}
\sum_{i=0}^{\lfloor \frac{k}{2}\rfloor-1}\binom{2i}{j}m_{2i}
=(\lambda_{j}-1)\binom{k}{j}
\ (j=0,1,\ldots,t).
\end{equation}
Since each $m_{2i}$ counts blocks, \eqref{eq:BI} must have
a solution in nonnegative integers.
For each $(t,v,k)$ in Table~\ref{Tab:Non0},
we consider $\lambda=\tau\mu$ with $1\leq\mu\leq\mu_{\max}$.
The column ``$\mu_{BI}$''  lists the values of $\mu$ for which \eqref{eq:BI} 
has no solution in nonnegative integers.
Thus, no self-orthogonal $t$-$(v,k,\tau\mu)$ design exists
for any value of $\mu$ listed in this column.
For example, when $(t,v,k)=(3,22,8)$, we have
$\tau=6$ and $\mu_{\max}=2$.
From~\eqref{eq:BI} we have $m_6=2\mu-4$,
so $\mu=1$ is excluded by the condition $m_6\geq0$.
Hence, the only possible $\lambda$ is $12$.

\begin{prop}\label{prop:nonex2}
There is no self-orthogonal $t$-$(v,k,\lambda)$ design
for
\[
(t,v,k)=
(3, 28,  8),
(3, 30,  8),
(3, 24, 10),
(4, 26, 12) \text{ and }
(4, 28, 12).
\]
\end{prop}
\begin{proof}
For each of the listed parameter sets,
Table~\ref{Tab:Non0} gives $\mu_{\max}=0$.
However, the existence of such a design requires
$\lambda=\tau\mu$ for a positive integer $\mu$
satisfying $\mu\leq\mu_{\max}$, a contradiction.
\end{proof}

\begin{prop}\label{prop:nonex3}
There is no self-orthogonal $t$-$(v,k,\lambda)$ design
for
\[
(t,v,k)=
(3, 26,  8),
(3, 28, 10),
(4, 30, 12) \text{ and }
(4, 28, 14).
\]
\end{prop}
\begin{proof}
For each of the listed parameter sets, Table~\ref{Tab:Non0}
shows that \eqref{eq:BI} has no nonnegative integer solution
for any $\lambda=\tau\mu$ with $1\leq\mu\leq\mu_{\max}$.
Since these are all possible values of $\lambda$,
the result follows.
\end{proof}

\begin{prop}
\begin{enumerate}
\item
There is a self-orthogonal $3$-$(22,8,\lambda)$ design
if and only if $\lambda=12$.
\item
There is a self-orthogonal $3$-$(24,  8, \lambda)$ design
if and only if $\lambda=21$.
\item
There is a self-orthogonal $4$-$(24,12,\lambda)$ design
if and only if $\lambda=120$.
\end{enumerate}

\bigskip
In each case, the design is unique, up to isomorphism.
\end{prop}
\begin{proof}
From Table~\ref{Tab:Non0}, the only possible values of $\lambda$
are $12$, $21$ and $120$ in (i), (ii) and (iii), respectively.
The existence and uniqueness up to isomorphism of the designs
with these parameters follow from~\cite{T86},
Remark~\ref{rem:24-48} and~\cite{T86d}, respectively.
\end{proof}

\begin{prop}\label{prop:22-10}
There is a self-orthogonal $3$-$(22,10,\lambda)$ design
if and only if $\lambda=48$.
The design is unique, up to isomorphism.
\end{prop}
\begin{proof}
Up to isomorphism, there is a unique self-orthogonal $3$-$(22, 10, 48)$ 
design~\cite{T86d}.
From Table~\ref{Tab:Non0}, it is sufficient to show
that there is no self-orthogonal $3$-$(22, 10, 6\mu)$ 
design for $\mu \in \{2,4,6\}$.

Suppose that $\cD$ is a self-orthogonal $3$-$(22,10,6\mu)$
design with $\mu\in\{2,4,6\}$.
By Lemma~\ref{lem:dual} (ii) and Lemma~\ref{3-des-v-2k},
$C(\cD)^\perp$ contains no vectors of weight $2$ or $4$.
Since $\allone\in C(\cD)^\perp$, two disjoint blocks
yield a vector of weight $2$ in $C(\cD)^\perp$,
where $\allone$ denotes the all-one vector.
Similarly, two blocks intersecting in $8$ points
yield a vector of weight $4$.
Hence, we have $m_0=m_8=0$ in \eqref{eq:BI}.
Multiplying these equations in \eqref{eq:BI} for $j=0,1,2,3$
by $-16,11,-6,2$, respectively, and adding, we obtain
\[
16(m_8-m_0)=8\mu-64.
\]
Thus, $\mu=8$, a contradiction.
\end{proof}
%%%%%%%%%%% 2-(20,10,lambda) %%%%%%%%%%%%%%
%%%%%%%%%%% 2-(20,10,lambda) %%%%%%%%%%%%%%

\begin{prop}\label{prop:20-10}
There is no self-orthogonal $3$-$(20,10,\lambda)$ design
for any positive integer $\lambda$.
\end{prop}

\begin{proof}
Let $\cD=(X,\cB)$ be a self-orthogonal $3$-$(20,10,\lambda)$ design.
By Lemma~\ref{lem:dual} (ii) and $d_{\max}(20)=4$,
$C(\cD)$ is contained in a self-dual $[20,10,4]$ code $C$.
Up to equivalence, there are $7$ such codes~\cite{Pless72}.

We verified by \textsc{Magma}~\cite{Magma} that
only one of these codes has the property that every $3$-subset
of the coordinate set is contained in the support of a codeword
of weight $10$.
This is the unique self-dual $[20,10,4]$ code with weight enumerator
\[
1+5y^4+80y^6+250y^8+352y^{10}
+250y^{12}+80y^{14}+5y^{16}+y^{20}.
\]
Since $\lambda>0$, $C$ must be this code, up to equivalence.
We also verified that some $3$-subset is contained in exactly
$32$ supports of codewords of weight $10$.
Hence, $\lambda\leq32$ and $|\cB|=19\lambda/2\leq304$.

Choose a codeword $y\in C$ of weight $10$ with $\supp(y)\notin\cB$.
Since $C\subseteq C(\cD)^\perp$, the Mendelsohn
equations~\eqref{eq:M} apply to $y$ and give
\[
n_4(y)=n_0(y)-4n_5(y)+\frac{3\lambda}{16}
      =n_0(y)+\frac{3\lambda}{16}.
\]
Adding $y$ to the incidence vectors of the blocks counted by $n_4(y)$
gives distinct codewords of weight $4$ in $C$.
Thus, we have $n_4(y)\leq5$.
Since $n_0(y)$ and $n_4(y)$ are nonnegative integers, we have
\[
16\mid\lambda \text{ and }0<\frac{3\lambda}{16}\leq5.
\]
Therefore, we have $\lambda=16$, that is,  
$\cD$ is a $3$-$(20,10,16)$ design.

We now show that no such design exists.
For a fixed block $B\in\cB$, \eqref{eq:BI} gives
\[
m_2=4-4m_0\text{ and } m_8=m_0-1.
\]
Since $m_0\in\{0,1\}$ and $m_8\geq0$, we have
$m_0=1$ and $m_2=m_8=0$.
Thus, $\cB$ is closed under complementation and
$|B\cap B'|\in\{0,4,6\}$ for any distinct $B,B'\in\cB$.
Let $\Gamma$ be the graph with vertex set
\[
V(\Gamma)=
\{\{x,\allone+x\}\mid x\in C,\;\wt(x)=10\},
\]
where two distinct vertices $\{x,\allone+x\}$ and
$\{y,\allone+y\}$ are adjacent if
$\wt(x+y)\in\{4,16\}$.
This definition is independent of the choice of representatives,
since replacing either representative by its complement
interchanges weights $4$ and $16$.
The following set
\[
\{\{x,\allone+x\}\mid x\in C,\;\supp(x)\in\cB\}
\]
is an independent set of $\Gamma$ of size $152/2=76$.
However, we verified by \textsc{Magma}~\cite{Magma} that
$\Gamma$ has independence number $70$, a contradiction.
\end{proof}

For $(t,v,k)$ given in Table~\ref{Tab:Non0}
(except the cases with $\mu_{\max}=0$),
we summarize in Table~\ref{Tab:Sum}
the existence and nonexistence results discussed in this paper
for self-orthogonal $t$-$(v,k,\mu\tau)$ designs.
Some of these results are established in
Section~\ref{sec:unsaturated}, as indicated in the table.
In the column ``$\mu_E$'', $\mu$
indicates that the existence of a self-orthogonal 
$t$-$(v,k,\mu \tau)$ design is known, and
in the column ``$\mu_N$'', $\mu$
indicates that 
the nonexistence of a self-orthogonal $t$-$(v,k,\mu \tau)$ design is
known.

%%%%%%%%%%%%%%%%%%%%%%%%%%%
\begin{table}[thb]
\caption{Self-orthogonal
$(\lfloor\frac{k}{4}\rfloor+1)$-$(v,k,\lambda)$ designs ($v \le 30, k=8,10,12,14$)}
\label{Tab:Sum}
\begin{center}
{\small
%{\footnotesize
%{\scriptsize
\begin{tabular}{c|c|c|c|c}
\noalign{\hrule height0.8pt}
$(t,v,k)$ & $\tau$ & $\mu_{\max}$& $\mu_{E}$&
 $\mu_{N}$\\
 \hline
$(3, 16,  8)$&   3 & 6& $1,2$ (Section~\ref{sec:unsaturated})& \\
$(3, 22,  8)$&   6 & 2&  2& $1$  \\
$(3, 24,  8)$&  21 & 1& 1 &  \\
$(3, 26,  8)$&   7 & 1& &  $1$  \\
$(3, 20, 10)$&   4 &13&   & $1,2,\ldots, 13$  \\
$(3, 22, 10)$&   6 & 9&  8 (\cite{T86d}) &$1,2,\ldots,7,9$  \\
$(3, 26, 10)$&   3 &20& & \\
$(3, 28, 10)$&  20 & 2&  & $1,2$  \\
$(3, 30, 10)$&  18 & 3&  & $1,2,3$ (Section~\ref{sec:unsaturated}) \\
$(4, 24, 12)$&  15 & 8& 8 & $1,2,\ldots,7$\\
$(4, 30, 12)$&  55 & 1& & $1$  \\
$(4, 28, 14)$&  66 & 3& &  $1,2,3$ \\
$(4, 30, 14)$& 143 & 2& & $1,2$ (Section~\ref{sec:unsaturated}) \\
\noalign{\hrule height0.8pt}
\end{tabular}
}
\end{center}
\end{table}
%%%%%%%%%%%%%%%%%

%%%%%%%%%%%%%%%%%%%%%%%%%%%%%%%%%%%%%%%%%%%
\section{Unsaturated self-orthogonal designs}
\label{sec:unsaturated}

In this section,
in order to show that self-orthogonal designs with certain parameters
are saturated,
we establish the nonexistence of unsaturated designs with the parameters.

\begin{lem}\label{lem:BS}
Let $\cD$ be a self-orthogonal $t$-$(v,k,\lambda)$ design
with $k$ even.
Suppose that $\cD$ is unsaturated.
Then there is a nonzero vector $u\in C(\cD)^\perp$
such that $|B\cap S|\leq k/2$ for every block $B$,
where $S=\supp(u)$.
\end{lem}

\begin{proof}
If $C(\cD)\ne C(\cD)^\perp$, then choose a vector $u$ of minimum weight
in a nontrivial coset of $C(\cD)$ in $C(\cD)^\perp$.
Otherwise, since $\cD$ is unsaturated,
\eqref{C2} or \eqref{C3} fails.
Hence, there is a nonzero codeword $u$ of minimum weight
in $C(\cD)$ whose support is not a block.

In either case, $u$ is a nonzero vector and $\supp(u)$ is not a block.
For any block $B$, let $\chi_B$ denote its incidence vector.
Then $u+\chi_B$ is nonzero and lies in the same coset of $C(\cD)$ as $u$.
By the choice of $u$, we have
\[
\wt(u)\leq\wt(u+\chi_B)
=\wt(u)+k-2|B\cap\supp(u)|.
\]
Thus, we conclude that $|B\cap\supp(u)|\leq k/2$.
\end{proof}

\begin{prop}\label{prop:BS}
Let $\cD$ be a self-orthogonal $t$-$(v,k,\lambda)$ design
with $k$ even.
Suppose that $\cD$ is unsaturated.
Let $t_0=\lfloor\frac{k}{4}\rfloor+1$ and
let $m=|S|$, where $S$ is given as in Lemma~\ref{lem:BS}.
Assume $t\geq t_0$.
Then
\begin{equation}\label{eq:mvD}
\sum_{i=1}^{t_0}
i(-2)^{i-1}\binom{2t_0-i-1}{t_0-1}
\binom{m}{i}
\prod_{j=i}^{t_0-1}\frac{v-j}{k-j}=0.
\end{equation}
\end{prop}

\begin{proof}
Since $S$ is the support of a vector in $C(\cD)^\perp$,
each block meets $S$ in an even number of points.
By Lemma~\ref{lem:BS}, this number is at most $k/2$.
Hence, the Mendelsohn equations~\eqref{eq:M} give
\begin{equation}\label{eq:unM}
\sum_{j=1}^{t_0-1}\binom{2j}{i}n_j
=\lambda_i\binom{m}{i}
\quad(i=1,2,\ldots,t_0).
\end{equation}
Applying Lemma~\ref{lem:Mendel} to~\eqref{eq:unM}
with $t=t_0$ and $b_i=\lambda_i\binom{m}{i}$, we obtain
\[
\sum_{i=1}^{t_0}
i(-2)^{i-1}\binom{2t_0-i-1}{t_0-1}
\lambda_i\binom{m}{i}=0.
\]
From~\eqref{eq:lambdass}, we have
\[
\lambda_i
=\lambda_{t_0}\prod_{j=i}^{t_0-1}\frac{v-j}{k-j}
\quad(i=1,2,\ldots,t_0).
\]
Substituting these expressions into the sum above
and dividing by $\lambda_{t_0}>0$ gives~\eqref{eq:mvD}.
\end{proof}

For $k=4$ or $6$, \eqref{eq:mvD} reduces to
\[
\frac{2m(v-(k-1)m+(k-2))}{k-1}=0.
\]
Under the conditions $k<v$ and $0<m<v$,
the integer solutions are
\begin{equation}\label{eq:un46}
(v,m)=((k-1)m-(k-2),m)\quad(m=3,4,\ldots).
\end{equation}
For $k=8,10,\ldots,30$, all integer solutions
of~\eqref{eq:mvD} satisfying
$k<v\leq10000$ and $0<m<v$
are listed in Table~\ref{Tab:un}.
This was found by \textsc{Magma}~\cite{Magma}.

%%%%%%%%%%%%%%%%%%%%%%%%%%%
\begin{table}[thb]
\caption{Solutions $(v,m)$ of~\eqref{eq:mvD} ($k=8,10,\ldots,30$)} 
\label{Tab:un}
\begin{center}
{\small
%{\footnotesize
%{\scriptsize
\begin{tabular}{c|l}
\noalign{\hrule height0.8pt}
$k$ & \multicolumn{1}{c}{$(v,m)$} \\
\hline
 8&$(16, 4),(16, 6),(22, 6),(22, 7),(23, 7)$\\
10&$(20, 4),(22, 6),(26, 6)$\\
12&$(23, 7),(23, 8),(24, 8),(36, 8),(47, 11),(68, 16),(156, 36),(311, 71)$\\
14&$(27, 7),(80, 16),(159, 31)$\\
16&$(22, 3),(24, 8),(27, 4),(32, 8),(43, 8),(43, 11), (48, 12),
    (56, 12),(58, 13)$\\
18&$(36, 8)$\\
20&$(39, 11)$\\
22&$(43, 11)$\\
24&$(48, 12),(72, 16),(96, 20)$\\
26&$(52, 12)$\\
28&$(55, 15)$\\
30&$(59, 15)$\\
\noalign{\hrule height0.8pt}
\end{tabular}
}
\end{center}
\end{table}
%%%%%%%%%%%%%%%%%%%%%%%%%%%%%%%%%%%%%%%%%%%%%%%%

We are now in a position to give a proof of
Theorem~\ref{thm:24-96}.
\begin{proof}[Proof of Theorem~\ref{thm:24-96}]
Suppose that $(v,k)=(24,8)$, $(48,12)$ or $(72,16)$.
Since these $(v,k)$ are not in Table~\ref{Tab:un},
a self-orthogonal $(\lfloor\frac{k}{4}\rfloor+1)$-$(v,k,\lambda)$ design
$\cD$ is saturated by Proposition~\ref{prop:BS}.
This means that $C(\cD)$ is an extremal doubly even self-dual code of length $v$.
Hence, by \eqref{C2} and the Assmus--Mattson theorem~\cite{AM},
$\cD$ is a $5$-$(v,k,\binom{5(v/24)-2}{(v/24)-1})$ design.
This completes a proof of
Theorem~\ref{thm:24-96}.
\end{proof}

We also apply Proposition~\ref{prop:BS} to have the following:

\begin{prop}\label{prop:nonex4}
There is no self-orthogonal $t$-$(v,k,\lambda)$ design
for
\[
(t,v,k)=(3,30,10)\text{ and }(4,30,14).
\]
\end{prop}

\begin{proof}
Suppose that $\cD$ is a self-orthogonal $t$-$(v,k,\lambda)$
design with $(t,v,k)=(3,30,10)$ or $(4,30,14)$.
If $\cD$ is unsaturated, then Proposition~\ref{prop:BS}
gives an integer solution of~\eqref{eq:mvD}
with $0<m<v$.
However, Table~\ref{Tab:un} contains no such solution
for $(v,k)=(30,10)$ or $(30,14)$.
Thus, $\cD$ is saturated.
Hence, $C(\cD)$ is a self-dual code of length $30$
and minimum weight $k\in\{10,14\}$.
This contradicts $d_{\max}(30)=6$.
\end{proof}

%%%%%%%%%%%%%%%%%%%%%%%%%%%
\begin{table}[thb]
\caption{$(\lfloor\frac{k}{4}\rfloor+1)$-$(v,k,\lambda)$ designs 
with parameters in Table~\ref{Tab:un} ($v$ even)}
\label{Tab:Non2}
\begin{center}
{\small
%{\footnotesize
%{\scriptsize
\begin{tabular}{c|c|c|c|rl|l|c}
\noalign{\hrule height0.8pt}
$(t,v,k)$ & $\tau$ & $\mu_{\max}$& $d(\cD)$ & \multicolumn{2}{c|}{$A_{(v,k)}$}
& \multicolumn{1}{c|}{$\mu_{BI}$} & $\mu_U$\\
\hline
$(4, 36, 12)$ & 15 &  5& $8$  & 10387 &(\cite{HM36}) & & \\
$( 5, 22, 16)$& 728&  0& $6$&   77 &(\cite{PS75}) & & \\
$(5, 24, 16)$ & 78 &  1& $6,8$&  759 &(\cite{PS75}) & & \\
$(5, 32, 16)$ & 39 & 20& $6,8$&36518 &(\cite{BR02}, \cite{CPS})&&$1,2,\ldots,20$\\
$(5, 36, 18)$ & 35 & 46& $6,8$&71232 &(\cite{HM36}) & $1,2,\ldots,31$ & $32$\\
              &    &   &      & && $33,\ldots,46$ & \\
\noalign{\hrule height0.8pt}
\end{tabular}
}
\end{center}
\end{table}
%%%%%%%%%%%%%%%%%

Next we  consider the parameter pairs $(v,k)$ given in Table~\ref{Tab:un}.
For the cases with even $v\leq36$ not included
in Table~\ref{Tab:Non0}, the corresponding bounds
are given in Table~\ref{Tab:Non2}.
Note that the numbers $A_{(v,k)}$ are known for $v\leq36$.
The column ``$\mu_{BI}$'' lists the values of $\mu$
for which the system~\eqref{eq:BI} with $\lambda=\tau\mu$
has no solution in nonnegative integers.
The column ``$\mu_U$'' lists the values of $\mu$
for which the existence of a self-orthogonal
$t$-$(v,k,\tau\mu)$ design is unknown.
From Table~\ref{Tab:Non2}, we have the following:

\begin{prop}\label{prop:nonex5}
\begin{enumerate}
\item
There is no self-orthogonal $5$-$(22,16,\lambda)$ design.
\item
There is a self-orthogonal $5$-$(24,16,\lambda)$ design
if and only if $\lambda=78$.
\item
If there is a self-orthogonal $5$-$(36,18,\lambda)$ design,
then $\lambda=1120$.
\end{enumerate}
\end{prop}

The existence of a self-orthogonal $5$-$(24,16,78)$ design
in (ii) follows by taking the complementary
design of the Witt system on $24$ points.

\begin{prop}\label{prop:3-36-12-15}
There is no self-orthogonal $4$-$(36,12,\lambda)$ design
for any positive integer $\lambda$.
\end{prop}
\begin{proof}
Let $\cD$ be a self-orthogonal $4$-$(36,12,15\mu)$ design.
From Table~\ref{Tab:Non2}, we have $\mu \in \{1,2,3,4,5\}$.
Since $C(\cD)$ is self-orthogonal, there is 
a self-dual code $C$ containing $C(\cD)$.
Since $36\not\equiv0\pmod8$, $C$ is not doubly even.
Hence, there is a vector $x\in C\subseteq C(\cD)^\perp$
of weight $2s$, where $s$ is odd.

We apply the Mendelsohn equations~\eqref{eq:M}
to the vector $x$ with $m=2s$.
Multiplying these equations for $i=1,2,3,4$
by $-5,10,-12,8$, respectively, adding them,
and dividing by $16$, we obtain
\[
8(n_4+5n_5+15n_6)
=5\mu s(s^3-14s^2+66s-104).
\]
Since $s$ is odd, $5s(s^3-14s^2+66s-104)$ is odd.
Thus, we have $8\mid\mu$, a contradiction.
\end{proof}

%%%%%%%%%%% nonexistence (1) %%%%%%%%%%%%%%
%%%%%%%%%%% nonexistence (1) %%%%%%%%%%%%%%
Now we examine the existence of
self-orthogonal $(\lfloor\frac{k}{4}\rfloor+1)$-$(v,k,\lambda)$ designs
for $(v,k)$ given in Table~\ref{Tab:un},
where $v$ is odd or $v$ is even with $v\geq38$.
Let $\cD=(X,\cB)$ be a self-orthogonal $t$-$(v,k,\lambda)$ design. 
Then $C(\cD)$ is a self-orthogonal code containing 
at least $|\cB|+1$ codewords.
Since $|\cB|+1 \le 2^{\lfloor v/2\rfloor}$,
we have
$\lambda \le 
\lfloor2^{\lfloor v/2\rfloor} \binom{k}{t}/\binom{v}{t}\rfloor$.
Let $\tau$ be as defined in Section~\ref{sec:4}.
Let
\[
\mu_{\max}=
\left\lfloor
\frac{2^{\lfloor \frac{v}{2}\rfloor} \binom{k}{t}}{\binom{v}{t} \tau}
\right\rfloor.
\]
Of course, 
if there is a self-orthogonal $t$-$(v,k,\mu \tau)$ design,
then $\mu \le \mu_{\max}$.
For these parameter pairs, Table~\ref{Tab:Non} lists
$\tau$ and $\mu_{\max}$, except for
$(t,v,k)=(4,156,12)$, $(4,311,12)$ and $(4,159,14)$.
For these three parameters, we have
\[
\tau(\mu_{\max}+1)>\binom{v-t}{k-t}.
\]
Thus, the bound $\mu\leq\mu_{\max}$ gives no additional
restriction beyond $\lambda\leq\binom{v-t}{k-t}$.

\begin{prop}\label{prop:4-27-16}
There is no self-orthogonal $5$-$(27,16,\lambda)$ design
for any positive integer $\lambda$.
\end{prop}
\begin{proof}
For $(t,v,k)=(5,27,16)$, Table~\ref{Tab:Non} gives
$\tau=2184$ and $\mu_{\max}=0$.
However, such a design has $\lambda=\tau\mu$
for a positive integer $\mu\leq\mu_{\max}$, a contradiction.
\end{proof}

%%%%%%%%%%%%%%%%%%%%%%%%%%%
\begin{table}[htb]
\caption{$(\lfloor\frac{k}{4}\rfloor+1)$-$(v,k,\lambda)$ designs 
with parameters in Table~\ref{Tab:un}}
\label{Tab:Non}
\begin{center}
{\small
%{\footnotesize
%{\scriptsize
\begin{tabular}{c|c|c||c|c|c}
\noalign{\hrule height0.8pt}
$(t,v,k)$ & $\tau$ & $\mu_{\max}$&
$(t,v,k)$ & $\tau$ & $\mu_{\max}$\\
\hline
$(3,23, 8)$&    8 &         8 &$(5,58,16)$&   52 &      9841 \\
$(4,23,12)$&    9 &        12 &$(6,39,20)$&  120 &        51 \\
$(4,47,12)$&    9 &      2586 &$(6,43,22)$&  561 &        45 \\
$(4,68,12)$&    9 &   1160253 &$(7,48,24)$&  285 &       276 \\
$(4,27,14)$&   77 &         6 &$(7,72,24)$& 1995 &      8092 \\
$(4,80,14)$&    2 & 347946717 &$(7,96,24)$&  209 &  39106891 \\
$(5,27,16)$& 2184 &         0 &$(7,52,26)$& 2310 &       142 \\
$(5,43,16)$&   56 &       169 &$(8,55,28)$& 1771 &       193 \\
$(5,48,16)$&  273 &       156 &$(8,59,30)$& 3450 &       410 \\
$(5,56,16)$&   42 &      7308 &&&\\
\noalign{\hrule height0.8pt}
\end{tabular}
}
\end{center}
\end{table}
%%%%%%%%%%%%%%%%%

Now let $\cD=(X,\cB)$ be a self-orthogonal $t$-$(v,k,\lambda)$ design
with $v$ odd. 
Let $C'(\cD)$ be the self-orthogonal code obtained
from $C(\cD)$ by appending a zero coordinate.
Then $C'(\cD)$ is contained in a self-dual code
of length $v+1$ containing at least $|\cB|$
codewords of weight $k$.
Let $A'_{(v+1,k)}$ denote the largest number of codewords of weight $k$ 
among self-dual codes of length $v+1$.
Let
\[
\mu'_{\max}=
\left\lfloor
\frac{A'_{(v+1,k)} \binom{k}{t}}{\binom{v}{t} \tau}
\right\rfloor.
\]
Of course, 
if there is a self-orthogonal $t$-$(v,k,\mu \tau)$ design,
then $\mu \le \mu'_{\max}$.
In the case $(t,v,k)=( 3, 23,  8)$, we have $\mu'_{\max}=3$, 
which improves the result in Table~\ref{Tab:Non},
where $A'_{(24,8)}=759$.

%%%%%%%%%%%%%%%%%%%%%%%%%%%
\begin{table}[htb]
\caption{$(\lfloor\frac{k}{4}\rfloor+1)$-$(v,k,\lambda)$ designs 
with parameters in Table~\ref{Tab:un} ($v$ odd)}
\label{Tab:NonII}
\begin{center}
{\small
%{\footnotesize
%{\scriptsize
\begin{tabular}{c|c|c|c}
\noalign{\hrule height0.8pt}
$(t,v,k)$ & $\tau$ & $\mu_{\max}$& \multicolumn{1}{c}{$\mu_{BI}$}\\
% & $\mu_U$\\
\hline
$( 3, 23,  8)$&    8 & 3 &              $1$               \\
$( 4, 23, 12)$&    9 & 12& $1,2,\ldots,7,9,\ldots,12$   \\
$( 4, 27, 14)$&   77 & 6 & $1,2,\ldots,6$               \\
$( 6, 39, 20)$&  120 & 51& $1,2,\ldots,31,33,\ldots,51$ \\
$( 6, 43, 22)$&  561 & 45& $1,2,\ldots,45$              \\
\noalign{\hrule height0.8pt}
\end{tabular}
}
\end{center}
\end{table}
%%%%%%%%%%%%%%%%%

For the cases in Table~\ref{Tab:Non} with $v$ odd
and $0<\mu_{\max}\leq100$, we give further restrictions
in Table~\ref{Tab:NonII}.
For $(t,v,k)=(3,23,8)$, we use the improved bound
$\mu'_{\max}=3$ in the column ``$\mu_{\max}$''.
The column ``$\mu_{BI}$'' has the same meaning
as in Table~\ref{Tab:Non0}.
From Table~\ref{Tab:NonII}, we have the following:

\begin{prop}\label{prop:nonex6}
\begin{enumerate}
\item
There is a self-orthogonal $4$-$(23,12,\lambda)$ design
if and only if $\lambda=72$.
\item
There is no self-orthogonal $t$-$(v,k,\lambda)$ design
for $(t,v,k)=(4,27,14)$ and $(6,43,22)$.
\item
If there is a self-orthogonal $6$-$(39,20,\lambda)$ design,
then $\lambda=3840$.
\end{enumerate}
\end{prop}

The existence of a self-orthogonal $4$-$(23,12,72)$ design
in (i) is established in~\cite{T86d}.

\begin{prop}\label{prop:23-8}
There is a self-orthogonal $3$-$(23,8,\lambda)$ design
if and only if $\lambda=16$.
The design is unique, up to isomorphism.
\end{prop}
\begin{proof}
Up to isomorphism, 
there is a unique self-orthogonal $3$-$(23, 8, 16)$ 
design~\cite{T86}.
From Table~\ref{Tab:NonII}, it is sufficient to show
that there is no self-orthogonal $3$-$(23, 8, 24)$ 
design. 

Let $\cD=(X,\cB)$ be a self-orthogonal
$3$-$(23,8,24)$ design.
Fix a block $B\in\cB$.
From~\eqref{eq:BI}, we have $m_6=2$.
Thus, there is a block $B'$ with $|B\cap B'|=6$.
The sum of the incidence vectors of $B$ and $B'$
is a vector of weight $4$ in
$C(\cD)\subseteq C(\cD)^\perp$.
This contradicts Lemma~\ref{lem:dual} (ii), which gives
$d(C(\cD)^\perp)\geq5$.
\end{proof}

\begin{prop}\label{prop:6-39-20}
There is no self-orthogonal $6$-$(39,20,\lambda)$ design
for any positive integer $\lambda$.
\end{prop}

\begin{proof}
By Proposition~\ref{prop:nonex6} (iii), it is sufficient to
show that there is no self-orthogonal $6$-$(39,20,3840)$ design.
Let $\cD=(X,\cB)$ be such a design.
Since $|\cB|=323232>2^{18}$ and $C(\cD)$ is self-orthogonal,
the dimension of  $C(\cD)$ is $19$.
In addition, the following code
\[
C=\left\langle
\{(x,0)\mid x\in C(\cD)\},\allone
\right\rangle
\]
is a doubly even self-dual code of length $40$.

Puncturing $C$ in the last coordinate gives $C(\cD)^\perp$.
By Lemma~\ref{lem:dual} (i), we have
\[
d(C(\cD)^\perp)
\geq 5+\frac{\lambda_5}{\lambda}
=5+\frac{34}{15}>7.
\]
Thus, $C$ is an extremal doubly even self-dual code
of length $40$.
In addition, $C$ has exactly $525504$
codewords of weight $20$~\cite{MS73}.

Since $\allone\in C$, these codewords occur in pairs
$\{x,\allone+x\}$, with exactly one member of each pair
having last coordinate $0$.
Hence, exactly $262752$ codewords of weight $20$
have last coordinate $0$.
The incidence vectors of the blocks, with a zero coordinate
appended, are among these codewords.
This contradicts $|\cB|=323232$.
\end{proof}

%%%%%%% unsaturated examples %%%%%%%
%%%%%%% unsaturated examples %%%%%%%
For some cases $(k,v,m)$ given in Table~\ref{Tab:un},
we mention the existence of an unsaturated 
self-orthogonal $(\lfloor\frac{k}{4}\rfloor+1)$-$(v,k,\lambda)$
design $\cD$ such that $C(\cD)^\perp$ contains 
a vector $u$ of weight $m$ given in Lemma~\ref{lem:BS}.

\begin{itemize}
\item 
$(k,v,m)=(8,16, 4)$ and $(8,16, 6)$:
There are exactly five nonisomorphic $3$-$(16,8,3)$ designs.
Since these designs are Hadamard $3$-designs,
these are self-orthogonal.
We denote the design corresponding to the Sylvester Hadamard matrix
by $\cD_{16,8,1}$.
It is known that the code $C(\cD_{16,8,1})$ is $R(1,4)$,
which is a self-orthogonal $[16,5,8]$ code.
In addition, 
we verified by \textsc{Magma}~\cite{Magma} that $63$ nontrivial cosets of $C(\cD_{16,8,1})$
in $C(\cD_{16,8,1})^\perp$ have minimum weights $4$ and $6$.

We constructed by \textsc{Magma}~\cite{Magma} a self-orthogonal $3$-$(16,8,6)$ design
$\cD_{16,8,2}=(X,\cB)$ as follows.
Let $G$ be the permutation group of order $96$
on $X=\{1,2,\ldots,16\}$ generated by
\begin{align*}
&(1, 16, 3, 10)(2, 15, 4, 9)(5, 12, 7, 14)(6, 11, 8, 13),\\
&(3, 6, 7)(4, 5, 8)(9, 15, 13, 10, 16, 14)(11, 12),\\
&(9, 10)(11, 12)(13, 14)(15, 16).
\end{align*}
The block set $\cB$ is the union of the $G$-orbits
of the following base blocks:
\begin{align*}
&    \{1, 2, 5, 6, 9, 10, 13, 14\},
    \{1, 4, 5, 8, 9, 12, 13, 16\},
\\
&    \{1, 3, 6, 8, 9, 11, 14, 16\},
    \{1, 2, 3, 4, 5, 6, 7, 8\},
\\
&    \{3, 4, 9, 10, 11, 12, 15, 16\},
    \{2, 4, 5, 8, 10, 12, 14, 15\},
\\
&
    \{3, 4, 5, 6, 9, 10, 13, 14\}.
\end{align*}
There are two inequivalent doubly even self-dual codes of length
$16$ and the two codes are denoted by $e_8^2$ and
$d_{16}^+$ (see~\cite[Table~I]{C-S}).
We verified by \textsc{Magma}~\cite{Magma}
that $C(\cD_{16,8,2})$ is equivalent to $d_{16}^+$.

These designs are unsaturated, failing~\eqref{C1}
for $\lambda=3$ and~\eqref{C3} for $\lambda=6$.

\item
$(k,v,m)=(8,22, 6)$ and $(8,22, 7)$:
We denote by $\cD_{22,8}$ the unique self-orthogonal 
$3$-$(22, 8, 12)$ design given in~\cite{T86}.
We verified by \textsc{Magma}~\cite{Magma} that
the code $C(\cD_{22,8})$ is a self-orthogonal $[22, 10, 8]$ code.
In addition, 
we verified by \textsc{Magma}~\cite{Magma} that
three nontrivial cosets of $C(\cD_{22,8})$ in $C(\cD_{22,8})^\perp$
have minimum weights $6$ and $7$.
Hence, $\cD_{22,8}$ is unsaturated, failing~\eqref{C1}.

\item
$(k,v,m)=(8,23, 7)$:
We denote by $\cD_{23,8}$ the unique self-orthogonal 
$3$-$(23, 8, 16)$ design given in~\cite{T86}.
The code $C(\cD_{23,8})$ is a self-orthogonal $[23, 11, 8]$ code,
which is equivalent to $G_{23}^\perp$. 
In addition, we verified by \textsc{Magma}~\cite{Magma} that
the unique nontrivial coset of $C(\cD_{23,8})$
in $C(\cD_{23,8})^\perp$ has minimum weight $7$.
Hence, $\cD_{23,8}$ is unsaturated, failing~\eqref{C1}.

\item
$(k,v,m)=(10,22, 6)$:
Up to isomorphism, 
there is a unique self-orthogonal $3$-$(22, 10, 48)$ design and
the code  is equivalent to $C_{22}$~\cite{T86d}.
Since $C_{22}$ has minimum weight $6$,
this design is unsaturated, failing~\eqref{C3}.

\item
$(k,v,m)=(12,23, 7)$ and $(12,23, 8)$:
We denote by $\cD_{23,12}$ the unique self-orthogonal 
$4$-$(23, 12, 72)$ design~\cite{T86d}.
We verified by \textsc{Magma}~\cite{Magma} that the code $C(\cD_{23,12})$ is equivalent to $G_{23}^\perp$.
As described above,
the unique nontrivial coset of $C(\cD_{23,12})$
in $C(\cD_{23,12})^\perp$ has minimum weight $7$.
Since $C(\cD_{23,12})$ has parameters $[23,11,8]$,
$\cD_{23,12}$ is unsaturated, failing both~\eqref{C1}
and~\eqref{C3}.

\item
$(k,v,m)=(12,24, 8)$:
Up to isomorphism, 
there is a unique self-orthogonal $4$-$(24, 12, 120)$ design 
and 
the code is equivalent to $G_{24}$~\cite{T86d}.
Since $G_{24}$ has minimum weight $8$,
this design is unsaturated, failing~\eqref{C3}.

\item
$(k,v,m)=(12,47, 11)$:
We denote by $\cD_{47,12}$ the unique self-orthogonal $4$-$(47, 12, 36)$ 
design~\cite{HMT05}.
We verified by \textsc{Magma}~\cite{Magma} that the code $C(\cD_{47,12})$ is equivalent to
$QR_{47}^\perp$, which is 
a self-orthogonal $[47, 23, 12]$ code.
In addition, we verified by \textsc{Magma}~\cite{Magma} that
the unique nontrivial coset of $C(\cD_{47,12})$
in $C(\cD_{47,12})^\perp$ has minimum weight $11$.
Hence, $\cD_{47,12}$ is unsaturated, failing~\eqref{C1}.

\item
$(k,v,m)=(16,24,8)$:
A self-orthogonal $5$-$(24, 16, 78)$ design $\cD_{24,16}$
is obtained as the complementary design of the Witt system on $24$ points.
Thus, the code $C(\cD_{24,16})$ is equivalent to $G_{24}$.
Since $G_{24}$ has minimum weight $8$,
$\cD_{24,16}$ is unsaturated, failing~\eqref{C3}.

\item
$(k,v,m)=(16,48,12)$:
We denote by $\cD_{48,16}$  the self-orthogonal $5$-$(48, 16, 1365)$ design,
whose blocks are the supports of the codewords
of weight $16$ in $QR_{48}$.
We verified by \textsc{Magma}~\cite{Magma} that the code $C(\cD_{48,16})$ is the same as $QR_{48}$.
Since $QR_{48}$ has minimum weight $12$,
$\cD_{48,16}$ is unsaturated, failing~\eqref{C3}.

\end{itemize}

\section{Self-orthogonal designs satisfying~\eqref{C3}}
\label{sec:minwtk}

In this section, we work with a slightly weaker
hypothesis than the assumption that $\cD$ is saturated.
We assume
that $C(\cD)$ has minimum weight $k$.

\begin{prop}\label{prop:1}
Let $\cD$ be a self-orthogonal
$t$-$(v,k,\lambda)$ design with $k$ even, and let
$t_0=\left\lfloor\frac{k}{4}\right\rfloor+1$.
If $C(\cD)$ has minimum weight $k$ and
\begin{equation*}%\label{eq:tk}
t\geq t_0\geq2,
\end{equation*}
then
\begin{equation}\label{eq:lambdas}
\lambda_s=\frac{%
(-1)^{t_0-1}2^{2t_0-1}t_0\binom{k/2}{k/2-t_0}}{
\sum_{i=1}^{t_0} 
i (-2)^{i-1}\binom{2t_0-i-1}{t_0-1}
\binom{k}{i}
\prod_{j=i}^{t_0-1}\frac{v-j}{k-j}}
\prod_{j=s}^{t_0-1}\frac{v-j}{k-j}
\end{equation}
is an integer for $s=0,1,\ldots,t_0$.
\end{prop}
\begin{proof}
Fix a block $B_0$ of $\cD$. 
Note that $\cD$ is a $t_0$-design.
Since $C(\cD)$ has minimum weight $k$, 
for any block $B$ we have $|B\cap B_0|\leq k/2$
unless $B=B_0$. Thus, we have
\[
n_j=0\text{ for }j>\left\lfloor\frac{k}{4}\right\rfloor\text{ and }j\neq \frac{k}{2},
\]
and obviously $n_{k/2}=1$.
From~\eqref{eq:M0}, we have
\begin{equation*}%\label{eq:saM}
\sum_{j=1}^{t_0-1} \binom{2j}{i} n_j 
=(\lambda_i-1)\binom{k}{i}
\ (i=1,2,\ldots,t_0).
\end{equation*}
By Lemma~\ref{lem:Mendel}, we have
\begin{equation*}%\label{eq:3}
\sum_{i=1}^{t_0} k_i(\lambda_i-1)\binom{k}{i}=0,
\end{equation*}
where $k_i$ is given in~\eqref{eq:kidef} with $t=t_0$.
From~\eqref{eq:lambdass}, we have
\[
\lambda_{t_0}\sum_{i=1}^{t_0} k_i\binom{k}{i}
\prod_{j=i}^{t_0-1}\frac{v-j}{k-j}
=\sum_{i=1}^{t_0} k_i\binom{k}{i}.
\]
By Lemma~\ref{lem:kk} with $s=t_0-1$ and $j=k/2-1$,
we have
\begin{align*}
\sum_{i=1}^{t_0} k_i\binom{k}{i}
&=k(-4)^{t_0-1}\binom{k/2-1}{t_0-1}\\
&=(-1)^{t_0-1}2^{2t_0-1}t_0
\binom{k/2}{k/2-t_0}\ne0.
\end{align*}
Thus, the coefficient of $\lambda_{t_0}$ is nonzero.
Solving for $\lambda_{t_0}$ and applying~\eqref{eq:lambdass}
gives~\eqref{eq:lambdas}.
\end{proof}

As a special case of~\eqref{eq:lambdas}, we have
\begin{equation}\label{eq:lambda}
\lambda_{t_0}=\frac{%
(-1)^{t_0-1}2^{2t_0-1}t_0\binom{k/2}{k/2-t_0}}{
\sum_{i=1}^{t_0} 
i (-2)^{i-1}\binom{2t_0-i-1}{t_0-1}
\binom{k}{i}
\prod_{j=i}^{t_0-1}\frac{v-j}{k-j}}.
\end{equation}

\begin{prop}\label{prop:2}
Let $\cD$ be a self-orthogonal 
$t$-$(v,k,\lambda)$ design with $k$ even and let
$t_0=\left\lfloor\frac{k}{4}\right\rfloor+1$.
If $C(\cD)$ has minimum weight $k$ and
\begin{equation*}%\label{eq:tk1}
t\geq t_0\geq2
\text{ and } k\leq 40,
\end{equation*}
then 
\begin{equation}\label{eq:list}
\begin{split}
&(t_0,v,k,\lambda_{t_0})=
(2,7,4,2),(2,8,4,3),
(2,9,4,6), % eliminated by the consideration of unsaturated case
\\ &\qquad
(2,16,6,2),
(2,21,6,4),(2,22,6,5),(2,24,6,10),%\\
(2,25,6,20),
\\ &\qquad
(3,16,8,3),(3,22,8,12),(3,23,8,16),(3,24,8,21),
\\ &\qquad
(3,26,8,28),(3,29,8,16),(3,30,8,12),(3,32,8,7),
\\ &\qquad
(3,26,10,3),(3,42,10,9),(3,46,10,8),(3,50,10,6),
\\ &\qquad
(4,36,12,15),(4,47,12,36),(4,48,12,44),(4,51,12,2640),
\\ &\qquad
(5,56,16,42),(5,64,16,91),(5,72,16,78),
\\ &\qquad
(7,120,24,231).
\end{split}
\end{equation}
\end{prop}
\begin{proof} 
For each even $k$ with $4\leq k\leq40$,
we verified by \textsc{Magma}~\cite{Magma} that
$v<\gamma(k)$ whenever $v>k$ and
$\lambda_{t_0}$ in~\eqref{eq:lambda} is a positive integer,
where $\gamma(k)$ is given in Table~\ref{Tab:list}.
Then we determined  by \textsc{Magma}~\cite{Magma} the set $V(k)$ of integers $v$
with $k<v<\gamma(k)$ for which
$\lambda_0,\lambda_1,\ldots,\lambda_{t_0}$
are positive integers.
The values of $\gamma(k)$ and $V(k)$ are listed
in Table~\ref{Tab:list}.
To illustrate our approach, 
we give details for $k=10$.
If $k=10$, then
\[
\lambda_{t_0}= \frac{1152}{v^2 - 84 v + 1892}.
\]
Since $\lambda_{t_0} \ge 1$, 
$v^2 - 84 v + 740=(v-10)(v-74) \le 0$.
Thus, we may take $\gamma(10) = 75$.
Checking all integers $v$ with $10<v<75$,
we verified by \textsc{Magma}~\cite{Magma} that
$V(10)=\{26,42,46,50\}$.
\end{proof}

\begin{rem}
The following cases 
\[
(t_0,v,k,\lambda_{t_0})=(3,26,8,28),
(3,30,8,12)\text{ and }(4,36,12,15)
\]
in~\eqref{eq:list}
are excluded by Propositions~\ref{prop:nonex3},
\ref{prop:nonex2} and \ref{prop:3-36-12-15}, respectively.
\end{rem}

%%%%%%%%%%%%%%%%%%%%%%%%%%%
\begin{table}[thb]
\caption{$\gamma(k)$ and $V(k)$ in the proof of Proposition~\ref{prop:2}} 
\label{Tab:list}
\begin{center}
{\small
%{\footnotesize
%{\scriptsize
\begin{tabular}{c|c|l||c|c|l}
\noalign{\hrule height0.8pt}
$k$ & $\gamma(k)$ & \multicolumn{1}{c||}{$V(k)$} &
$k$ & $\gamma(k)$ & \multicolumn{1}{c}{$V(k)$} \\
\hline
 4& 10& $\{7,8,9\}$                           &24&151& $\{120\}$\\
 6& 26& $\{16,21,22,24,25\}$                  &26&195& $\emptyset$\\
 8& 45& $\{16, 22, 23$, $24, 26, 29, 30, 32\}$&28&139& $\emptyset$\\
10& 75& $\{26,42,46, 50\}$                    &30&158& $\emptyset$\\
12& 52& $\{36, 47, 48, 51\}$             &32&224& $\emptyset$\\
14& 70& $\emptyset$                      &34&256& $\emptyset$\\
16&104& $\{56,64,72\}$                &36&183& $\emptyset$\\
18&  1& $\emptyset$                     &38&202& $\emptyset$\\
20& 95& $\emptyset$                     &40&284& $\emptyset$\\
22&114& $\emptyset$ &&&\\
\noalign{\hrule height0.8pt}
\end{tabular}
}
\end{center}
\end{table}
%%%%%%%%%%%%%%%%%%%%%%%%%%%%%%%%%%%%%%%%%%%%%%%%

%%%%%%%%%%%%%%%%%%%%%%%%%%%%%%%%%%%%%%%%%%%%%%%%%%%%%%%
\section{Self-orthogonal codes and designs satisfying~\eqref{C2}
and \eqref{C3}} \label{sec:code}

The aim of this section is to give a proof of Theorem~\ref{thm:code}.

Since $k$ is even and $t<k$, the assumptions
$2\leq t\leq11$ and $t\geq\lfloor k/4\rfloor+1$
imply $4\leq k\leq42$.
For $k=42$, we have $t_0=11$.
We verified by \textsc{Magma}~\cite{Magma} that
$\lambda_{11}$ in~\eqref{eq:lambda} is not a positive integer
for any integer $v$ with $42<v<316$.
Write $\lambda_{11}=N/D(v)$, where $N>0$ is the numerator
and $D(v)$ is the denominator in~\eqref{eq:lambda}.
We also verified by \textsc{Magma}~\cite{Magma}
that all coefficients of $D(316+x)-N$ are positive.
Hence, $D(v)>N>0$ for all $v\geq316$, and therefore
$0<\lambda_{11}<1$.
Thus, Proposition~\ref{prop:1} excludes the case $k=42$.

The cases $k=8,10,16,18$ are covered by
Theorem~\ref{thm:Labayle}.
Therefore, we classify self-orthogonal
$(\lfloor k/4\rfloor+1)$-$(v,k,\lambda)$ designs
satisfying~\eqref{C2} and \eqref{C3}
for the remaining even values of $k$ with $4\leq k\leq40$.

\begin{lem}\label{lem:C2C3}
Let $\cD$ be a self-orthogonal
$t$-$(v,k,\lambda)$ design with $k$ even.
Then $\cD$ satisfies~\eqref{C2} and \eqref{C3}
if and only if $C(\cD)$ has minimum weight $k$
and the supports of its codewords of weight $k$
are precisely the blocks of $\cD$.
\end{lem}
\begin{proof}
This follows directly from~\eqref{C2} and \eqref{C3}.
\end{proof}

\begin{rem}
Let $C$ be a self-orthogonal code of minimum weight $k$
whose codewords of minimum weight form a $t$-design $\cD$.
Then $\cD$ satisfies~\eqref{C2} and \eqref{C3},
and $C(\cD)\subseteq C$, but the inclusion may be strict.
Thus, to classify such codes $C$, we must also determine
whether there are self-orthogonal codes properly containing
$C(\cD)$ with minimum weight $k$ and exactly the same
codewords of weight $k$.
\end{rem}

Let $\cD$ be a self-orthogonal
$(\lfloor k/4\rfloor+1)$-$(v,k,\lambda)$ design satisfying
\eqref{C2} and \eqref{C3},
where $k$ is even with $4\leq k\leq40$.
Set $t=\lfloor k/4\rfloor+1$.
By Proposition~\ref{prop:2},
$(t,v,k,\lambda)$ is listed in~\eqref{eq:list}.
For $(v,k)=(9,4),(25,6),(29,8)$ and $(51,12)$,
it follows from \eqref{eq:un46} and Table~\ref{Tab:un}  that
\eqref{eq:mvD} has no integer solution with $0<m<v$.
By Proposition~\ref{prop:BS}, $\cD$ must be saturated,
which is impossible since $v$ is odd.
The following cases
\[
(t,v,k,\lambda)=(3,26,8,28),(3,30,8,12)
\text{ and }(4,36,12,15)
\]
are excluded by Propositions~\ref{prop:nonex3},
\ref{prop:nonex2} and \ref{prop:3-36-12-15}, respectively.
We consider the remaining cases in~\eqref{eq:list} below.

\begin{itemize}
\item $(t,k)=(2,4)$:
The remaining parameters are $(v,\lambda)=(7,2)$ and $(8,3)$,
for which $|\cB|=7$ and $14$, respectively.
Since
\[
|\cB|+1\leq |C(\cD)|\leq2^{\lfloor \frac{v}{2}\rfloor},
\]
$C(\cD)$ has parameters $[7,3,4]$ and $[8,4,4]$, respectively.
These codes are unique up to equivalence and are
$e_7^\perp$ and $e_8$, respectively.
From~\eqref{C2}, $\cD$ is therefore isomorphic to
the unique $2$-$(7,4,2)$ design or
the unique $3$-$(8,4,1)$ design, respectively.
Both codes are maximal self-orthogonal.

\item $(t,k)=(2,6)$:
Distinct blocks meet in $0$ or $2$ points, since an intersection
of size $4$ yields a codeword of weight $4$,
contrary to~\eqref{C3}.
Thus, $\cD$ is symmetric when $v=16$ and quasi-symmetric otherwise.

For $v=16$, let $B_1,B_2$ and $B_3$ denote the three nonisomorphic
symmetric $2$-$(16,6,2)$ designs listed
in~\cite[Section~15.8]{Hall}\footnote{The tenth block of $B_1$
should be $\{2,6,10,12,14,15\}$.}.
We verified by \textsc{Magma}~\cite{Magma} that
only $B_1$ satisfies~\eqref{C2} and \eqref{C3},
and that $C(B_1)$ has parameters $[16,6,6]$.
This self-orthogonal code is unique up to equivalence,
and no self-orthogonal $[16,7,6]$ code exists~\cite{BBGO}.
Hence, no self-orthogonal code of minimum weight $6$
properly contains $C(B_1)$.

For $v=21$, there is a unique quasi-symmetric
$2$-$(21,6,4)$ design, up to isomorphism.
It satisfies~\eqref{C2} and \eqref{C3}, and its code is
equivalent to a $[21,10,6]$ code obtained by shortening
$C_{22}$ in one coordinate~\cite{T86b}.
All such shortened codes are equivalent.

For $v=22$, every quasi-symmetric $2$-$(22,6,5)$ design
is isomorphic to the unique $3$-$(22,6,1)$ design,
which satisfies~\eqref{C2} and \eqref{C3} and generates
a code equivalent to $C_{22}$~\cite{T86b}.
Both the $[21,10,6]$ code and $C_{22}$ are maximal
self-orthogonal.

For $v=24$, Proposition~\ref{prop:BS} and~\eqref{eq:un46}
imply that $\cD$ is saturated.
Hence, $C(\cD)$ is equivalent to $C_{24}$,
which has only $64$ codewords of weight $6$~\cite{PS75},
contradicting $|\cB|=184$.

\item $(t,k)=(4,12)$:
For $(v,\lambda)=(47,36)$ and $(48,44)$,
there is a unique self-orthogonal $4$-$(v,12,\lambda)$ design,
up to isomorphism.
These designs satisfy~\eqref{C2} and \eqref{C3},
and their codes are equivalent to
$QR_{47}^\perp$ and $QR_{48}$, respectively~\cite{HMT05}.
Both codes are maximal self-orthogonal.

\item $(t,k)=(3,8),(3,10)$ and $(5,16)$:
These cases are covered by Theorem~\ref{thm:Labayle}.
\end{itemize}

\begin{prop}\label{prop:7-120}
There is no self-orthogonal $7$-$(120,24,231)$ design.
\end{prop}

\begin{proof}
Let $\cD$ be a self-orthogonal $7$-$(120,24,231)$ design.
By Proposition~\ref{prop:BS} and Table~\ref{Tab:un},
$\cD$ is saturated.
Thus, $C(\cD)$ is an extremal doubly even self-dual
$[120,60,24]$ code whose codewords of minimum weight
have supports forming $\cD$.
This contradicts~\cite[Theorem~5]{BKST},
since $\cD$ is also a $6$-design.
\end{proof}

We are now in a position to give a proof of
Theorem~\ref{thm:code}.

\begin{proof}[Proof of Theorem~\ref{thm:code}]
Let $C$ be as in Theorem~\ref{thm:code}, and suppose that
$2\leq t\leq11$.
Then $4\leq k\leq42$.
By Lemma~\ref{lem:C2C3} and the above arguments,
we obtain a classification of self-orthogonal codes
of minimum weight $k$ whose codewords of minimum weight
form a $(\lfloor k/4\rfloor+1)$-design,
for $4\leq k\leq42$.
For $t=2,3,4$ and $5$, the assertions of Theorem~\ref{thm:code} 
follow from this classification, the known design properties
of these codes and the integrality conditions
in~\eqref{eq:lambdass}.

Suppose that $6\leq t\leq11$.
The above classification gives $k\leq16$.
Since a $t$-design is also a $5$-design,
Theorem~\ref{thm:Labayle} (ii) implies that $C$ is equivalent
to one of $G_{24}$, $QR_{48}$ and $C_{72}$.
However, the supports of their minimum-weight codewords
do not form $6$-designs by Remark~\ref{rem:24-48} (ii).
This completes the proof of Theorem~\ref{thm:code}.
\end{proof}

In addition, as a consequence of Theorem~\ref{thm:code},
we have the following classification of 
certain saturated designs.

\begin{cor}\label{cor:saturated}
Let $\cD$ be a self-orthogonal $t$-$(v,k,\lambda)$ design,
where $k$ is even, $4\leq k\leq40$ and $t\geq\lfloor k/4\rfloor+1$.
Then $\cD$ is saturated if and only if
$\cD$ is isomorphic to one of the following:
\begin{enumerate}
\item 
the unique $3$-$(8,4,1)$ design,
\item 
the unique $3$-$(22,6,1)$ design,
\item 
the unique $5$-$(24,8,1)$ design,
\item 
the five nonisomorphic self-orthogonal $3$-$(32,8,7)$ designs,
\item 
the unique self-orthogonal $3$-$(46,10,8)$ design,
\item 
the unique self-orthogonal $5$-$(48,12,8)$ design\text{ and}
\item 
a putative self-orthogonal $5$-$(72,16,78)$  design.
\end{enumerate}
\end{cor}
\begin{proof}
If $\cD$ is saturated, then $C(\cD)$ is self-dual
and has minimum weight $k$.
Applying Theorem~\ref{thm:code} with $t=\lfloor k/4\rfloor+1$
and using~\eqref{C2} gives the designs listed above.
Conversely, every listed design satisfies
\eqref{C1}, \eqref{C2} and \eqref{C3}
(see~\cite{HKM04}, \cite{HMT05}, \cite{T86b} and \cite{T86c}).
\end{proof}

%%%%%%%%%%%%%%%%%%%%
\bigskip
\noindent
\textbf{Acknowledgments.}
This work was supported by JSPS KAKENHI Grant Number 23K25784.
The results presented in this paper were mainly obtained during discussions with 
Akihiro Munemasa around 2014 and 2015. 
The author is deeply grateful to him for his useful suggestions, continuous discussions
and for allowing the author to publish this paper as a single author.

%%%%%%%%%%%%%%%%%%%%%%%%%%%%%%
\bigskip
\noindent
\textbf{Use of generative AI.}
In October 2026, while preparing this paper for submission,
the author used OpenAI's ChatGPT (GPT-6 Astra) to assist in 
checking the proofs, simplifying arguments and improving the exposition.

%%%%%%%%%%%%%%%%%  References  %%%%%%%%%%%%%%%%%%%%%%%%


\begin{thebibliography}{99}

\bibitem{AM} E.F. Assmus, Jr.\ and H.F. Mattson, Jr.,
{New $5$-designs},
\textsl{J.\ Combin.\ Theory}
\textbf{6} (1969), 122--151.

\bibitem{BKST}
E. Bannai, M. Koike, M. Shinohara and M. Tagami,
Spherical designs attached to extremal lattices and the modulo $p$ 
property of Fourier coefficients of extremal modular forms, 
\textsl{Mosc.\ Math.\ J.} \textbf{6}  (2006), 225--264.

\bibitem{BR02}R.T. Bilous and G.H.J. van Rees,
An enumeration of binary self-dual codes of length $32$,
\textsl{Des.\ Codes Cryptogr.}
\textbf{26} (2002), 61--86.

\bibitem{Magma}W. Bosma, J. Cannon and C. Playoust,
The Magma algebra system I: The user language,
\textsl{J. Symbolic Comput.}
\textbf{24} (1997), 235--265.

\bibitem{BBGO}
I. Bouyukliev, S. Bouyuklieva, T.A. Gulliver and P.R.J. \"Osterg\aa rd,
Classification of optimal binary self-orthogonal codes,
\textsl{J. Combin.\ Math.\ Combin.\ Comput.}
\textbf{59} (2006), 33--87.

\bibitem{C-S} J.H.~Conway and N.J.A.~Sloane,
A new upper bound on the minimal distance of self-dual codes,
\textsl{IEEE\ Trans.\ Inform.\ Theory}
\textbf{36} (1990), 1319--1333.

\bibitem{CPS} J.H. Conway, V. Pless and N.J.A. Sloane,
{The binary self-dual codes of length up to $32$: a revised enumeration},
\textsl{J.\ Combin. Theory Ser.~A}
\textbf{60} (1992), 183--195.

\bibitem{CW12} J. de la Cruz and W. Willems, 
$5$-designs related to binary extremal self-dual codes of
length $24m$,
Theory and applications of finite fields, 75--80, 
Contemp. Math., 579, Amer. Math. Soc., Providence, RI, 2012. 

\bibitem{Hall}M. Hall, Jr.,
Combinatorial Theory (second edition),
John Wiley \& Sons, New York, 1986.

\bibitem{H05} M. Harada, 
Remark on a $5$-design related to a putative extremal doubly-even 
self-dual $[96,48,20]$ code,
\textsl{Des.\ Codes Cryptogr.}
\textbf{37} (2005), 355--358.

\bibitem{HKM04} M. Harada, M. Kitazume and A. Munemasa, 
On a $5$-design related to an extremal doubly even self-dual code of length $72$,
\textsl{J. Combin.\ Theory Ser.~A}
\textbf{107} (2004), 143--146.

\bibitem{HM36} M. Harada and A. Munemasa,
Classification of self-dual codes of length $36$,
\textsl{Adv. Math. Commun.}
\textbf{6} (2012), 229--235.


\bibitem{HMT05} M. Harada, A. Munemasa and V.D. Tonchev, 
A characterization of designs related to an
extremal doubly-even self-dual code of length $48$,
\textsl{Ann.\ Comb.}
\textbf{9} (2005), 189--198.


\bibitem{Labayle}M. Lalaude-Labayle, 
On binary linear codes supporting $t$-designs,
\textsl{IEEE Trans.\ Inform.\ Theory}
\textbf{47} (2001), 2249--2255. 

\bibitem{MS73}C.L. Mallows and N.J.A. Sloane, 
An upper bound for self-dual codes, 
\textsl{Inform.\ Control}
\textbf{22} (1973), 188--200. 

\bibitem{M71} N.S. Mendelsohn, 
Intersection numbers of $t$-designs,
In: Studies in Pure Mathematics (presented
to Richard Rado), Academic Press, London 1971, 145--150.

\bibitem{Pless72}V. Pless, 
A classification of self-orthogonal codes over GF($2$),
\textsl{Discrete Math.} 
\textbf{3} (1972), 209--246.

\bibitem{PS75}V. Pless and N.J.A. Sloane, 
On the classification and enumeration of self-dual codes,
\textsl{J.\ Combin. Theory Ser.~A}
\textbf{18}  (1975), 313--335.

\bibitem{T86} V.D. Tonchev,
{A characterization of designs related to the Witt system $S(5,8,24)$},
\textsl{Math.\ Z.}
\textbf{191} (1986), 225--230.

\bibitem{T86b} V.D. Tonchev,
Quasi-symmetric designs and self-dual codes,
\textsl{Europ.\ J. Combin.}
\textbf{7} (1986), 67--73.

\bibitem{T86c} V.D. Tonchev,
Quasi-symmetric $2$-$(31,7,7)$ designs and a revision of 
Hamada's conjecture,
\textsl{J. Combin.\ Theory Ser.~A}
\textbf{42}  (1986),  104--110.

\bibitem{T86d} V.D. Tonchev,
A characterization of designs related to dodecads in the Witt 
system $S(5,8,24)$, 
\textsl{J. Combin.\ Theory Ser.~A}
\textbf{43}  (1986),  219--227.

\end{thebibliography}
\end{document}